\documentclass[psamsfonts]{amsart}

\usepackage{amssymb,amsfonts}
\usepackage[all,arc]{xy}
\usepackage{enumerate}
\usepackage{mathrsfs}
\usepackage{amsmath}
\usepackage{bbm}
\usepackage{tikz}
\usepackage{pgfplots}
\usepackage{hyperref}
\usepackage[alphabetic]{amsrefs}
\usepackage{mathtools}

\newtheorem{thm}{Theorem}[section]
\newtheorem{cor}[thm]{Corollary}
\newtheorem{prop}[thm]{Proposition}
\newtheorem{lem}[thm]{Lemma}

\theoremstyle{definition}
\newtheorem{defn}[thm]{Definition}

\theoremstyle{remark}

\newcommand{\e}{\varepsilon}

\newcommand{\C}{\mathbb{C}}

\newcommand{\n}{{\bf n}}

\newcommand{\x}{{\bf x}}
\newcommand{\X}{{\bf X}}
\newcommand{\y}{{\bf y}}

\newcommand{\z}{{\bf z}}
\newcommand{\Z}{\mathbb{Z}}
\newcommand{\N}{\mathbb{N}}
\renewcommand{\P}{\mathbb{P}}
\renewcommand{\mod}{\text{ mod }}

\newcommand{\E}{\mathbb{E}}

\newcommand{\ch}{\mathbbm{1}}
\newcommand{\vertiii}[1]{{\left\vert\kern-0.25ex\left\vert\kern-0.25ex\left\vert #1 
		\right\vert\kern-0.25ex\right\vert\kern-0.25ex\right\vert}}

\begin{document}
	\title{A Dynamical Proof of the Prime Number Theorem}
	\author{Redmond McNamara}
	\maketitle
	%	\end{center}
	\begin{abstract}
		We present a new, elementary, dynamical proof of the prime number theorem.
	\end{abstract}

	\section{Introduction}
	The prime number theorem states that
	\[
	\pi(N) = (1 + o_{N \rightarrow \infty}(1)) \frac{N}{\log N},
	\]
	where $\pi(N)$ denotes the number of primes of size at most $N$. 
	In some sense, the result was first publicly conjectured by Legendre in 1798 who suggested that
	\[
	\pi(N) = \frac{N}{A \log N + B + o_{N \rightarrow \infty}(1)},
	\]
	for some constants $A$ and $B$. Legendre specifically conjectured $A = 1$ and $B = -1.08366$. Gauss conjectured the same formula and stated he was not sure what the constant $B$ might turn out to be. Gauss' conjecture was based on millions of painstaking calculations first obtained in 1792 and 1793 which were never published but nonetheless predate Legendre's work on the subject. It is worth noting that later in his 1849 letter to Encke Gauss conjectured that $\pi(N) \approx \text{Li}(N)$, which in particular implies the correct values for $A$ and $B$. The first major breakthrough on the problem was due to Chebyshev who showed that
	\[
	c + o_{N \rightarrow \infty}(1) \leq \frac{ \pi(N)  \log N}{N} \leq C + o_{N \rightarrow \infty}(1)
	\]
	for some explicit constants $c$ and $C$ with $c > 0$. There is a long history of improvements to these explicit constants for which we refer to Goldstein \cite{Goldstein} and Goldfeld \cite{Goldfeld}.
	The prime number theorem was important motivation for Riemann's seminal work on the zeta function.
	
	The first proofs of the prime number theorem were given independently by Hadamard and de la Vall\'ee Poussin in 1896. The key step in their proof is a difficult argument showing that the Riemann zeta function does not have a zero on the line Re$(z) = 1$. Their proof was later substantially simplified by many mathematicians.
	In 1930,  Wiener found a ``Fourier analytic'' proof of the prime number theorem.
	In 1949, Erd\"os \cite{Erdos} and Selberg \cite{Selberg} discovered an elementary proof of the prime number theorem, where here elementary is used in the technical sense that the proof involves no complex analysis and does not necessarily mean that the proof is easy reading. The bitter battle over credit for this result is the subject of an informative note by Goldfeld \cite{Goldfeld}. 
	Other proofs are due to Daboussi \cite{Daboussi} and Hildebrand \cite{Hildebrand}. 
	In a blog post from 2014, Tao proves the prime number theorem using the theory of Banach algebras \cite{taoblog3}. A published version of this theorem can be found in a book by Einsiedler and Ward \cite{EW}.
	In an unpublished book from 2014, Granville and Soundarajan prove the prime number theorem using pretentious methods (see, for instance, \cite{GHS}).
	A note by Zagier \cite{zagier} from 1997 contains perhaps the quickest proof of the prime number theorem using a tauberian argument in the spirit of the Erd\"os-Selberg proof combined with complex analysis in the form of Cauchy's theorem. Zagier attributes this proof to Newman.
	
	The goal of this note is to present a new proof of the prime number theorem. Florian Richter and I discovered similar proofs concurrently and independently. His proof can be found in \cite{florian}. Terence Tao wrote up a version of this argument on his blog following personal communication from the author which can be found in \cite{taoblog2}.
	
	The proof proceeds as follows. To prove the prime number theorem, it suffices to prove that
	\[
	\frac{1}{N} \sum_{n \leq N} \Lambda(n) = 1 + o(1),
	\]
	where $\Lambda(n)$ is the von Mangoldt function which is $\log p$ if $n$ is a power of a prime $p$ and $0$ otherwise. The reader may think of $\Lambda$ as the normalized indicator function of the primes. The von Mangoldt function is related to the M\"obius function via the  formula
	\[
	\Lambda = \mu * \log,
	\]
	where the M\"obius function $\mu(n)$ is $0$ if $n$ has a repeated factor, $-1$ if $n$ has an odd number of distinct prime factors, $+1$ if $n$ has an even number of distinct prime factors. This formula, sometimes called the M\"obius inversion formula, encodes the fundamental theorem of arithmetic. Thus, there is a dictionary between properties of the von Mangoldt function $\Lambda$ and the M\"obius function $\mu$. Landau observed that cancellation in the M\"obius function is equivalent to the prime number theorem i.e. the prime number theorem is equivalent to the statement
	\[
	\frac{1}{N} \sum_{n \leq N} \mu(n) = o_{N \rightarrow \infty}(1).
	\]
	This is what we actually try to prove.
	
	The next observation is that, if one wants to compute a sum, it suffices to sample only a small number of terms. Typically (for instance for an i.i.d. randomly chosen sequence) the average value
	\begin{align*}
	&\frac{1}{N} \sum_{n \leq N} a(n)
	\intertext{
		is approximately the same as the average over only the even terms
	}
	\approx &\frac{2}{N} \sum_{n \leq N} a(n) \ch_{2 | n}.
	\end{align*}
	However, for certain sequences, like $a = (-1, +1, -1, +1,\ldots)$, the averages do not agree. Still for this sequence, if we instead sample every third point or every fifth point or every $p^{th}$ point for any other prime then the averages are approximately equal. It turns out, this is a rather general phenomenon: for any sequence, for most primes $p$, the average of the sequence is the same as the average along only those numbers divisible by $p$.
	
	Applying this to the M\"obius function, for each $N$, for most primes $p$
	\[
	\frac{1}{N} \sum_{n \leq N} \mu(n) \approx \frac{1}{N} \sum_{n \leq N} \mu(n) p \ch_{p | n}.
	\]
	For the purposes of this introduction, we will ``cheat" and pretend that this equation is true for any prime $p$. By changing variables
	\[
	\frac{1}{N} \sum_{n \leq N} \mu(n) p \ch_{p | n} = \frac{p}{N} \sum_{n \leq N/p} \mu(pn).
	\]
	But $\mu(pn) = -\mu(n)$ for most numbers $n$ since $\mu$ is multiplicative. Combining the last two equations gives
	\[
	\frac{1}{N} \sum_{n \leq N} \mu(n) \approx -\frac{p}{N} \sum_{n \leq N/p} \mu(n).
	\]
	The plan is to use this identity three times. Suppose we can find primes $p_1$, $p_2$ and $p$ such that $\frac{p_1p_2}{p} \approx 1$. Then by applying the previous identity
	\[
	\frac{1}{N} \sum_{n \leq N} \mu(n) \approx -\frac{p}{N} \sum_{n \leq N/p} \mu(n)
	\]
	and also
	\begin{align*}
	\frac{1}{N} \sum_{n \leq N} \mu(n) \approx & -\frac{p_1}{N} \sum_{n \leq N/p_1} \mu(n) \\
	\approx & +\frac{p_1 p_2}{N} \sum_{n \leq N/p_1 p_2} \mu(n).
	\end{align*}
	But since $\frac{p_1 p_2}{p} \approx 1$, we know that

\[
	\frac{p_1 p_2}{N} \sum_{n \leq N/p_1 p_2} \mu(n) \approx \frac{p}{N} \sum_{n \leq N/p} \mu(n).
	\]
	Putting everything together we conclude that
	\begin{align*}
	\frac{1}{N} \sum_{n \leq N} \mu(n) \approx & -\frac{1}{N} \sum_{n \leq N} \mu(n) 
	\intertext{
		which implies
	}
	\frac{1}{N} \sum_{n \leq N} \mu(n) \approx & \ 0.
	\end{align*}
	This implies the prime number theorem.
	
	Thus, the main difficulty in the proof is finding primes $p$, $p_1$ and $p_2$ lying outside some exceptional set for which $\frac{p_1p_2}{p} \approx 1$. We give a quick sketch of the argument. The Selberg symmetry formula (Theorem \ref{SelbergSymmetryFormula}) roughly tells us that, even if we do not know how many 
	primes there are at a certain scale (say in the interval from $x$ to $x(1+\e)$) and we do not know how many semiprimes (products of two primes) there are at that scale, the weighted sum of the number of primes and semiprimes is as we would expect. In particular, if there are no semiprimes between $x$ and $x(1+\e)$ there are twice as many primes as one would expect (meaning $2 \cdot \e \frac{x}{\log x}$ many primes). Let $x$ be a large number. If there are both primes and semiprimes between $x$ and $x(1+\e)$ then we can find $p$, $p_1$ and $p_2$ such that $\frac{p_1p_2}{p} \approx 1 + O(\e)$ and we are done. Thus, assume that there are either only primes or only semiprimes in the interval $[x, x(1+\e)]$. For the sake of our exposition, we will assume there are only primes between $x$ and $x(1+\e)$. By the Selberg symmetry formula, there are twice as many primes in this interval as expected. Now if there is a semiprime $p_1 p_2$ in the interval $[x(1+\e), x(1+\e)^2]$ then picking any prime $p$ in the interval $[x, x(1+\e)]$ we conclude that there exists $p$, $p_1$ and $p_2$ such that 
	$\frac{p_1p_2}{p} \approx 1 + O(\e)$. Thus, either we win (and the prime number theorem is true) or there are again twice as many primes in the interval $[x(1+\e), x(1+\e)^2]$ as one would expect. Running this argument again shows that there are again only primes and no semiprimes in the interval $[x(1+\e)^2, x(1+\e)^3]$. Iterating this argument using the connectedness of the interval, we find large intervals $[x, 100x]$ where there are twice as many primes as predicted by the prime number theorem. But this contradicts Chebyshev's theorem: Chebyshev's theorem gives a lower bound on the number of primes, which in turn gives a lower bound on the number of semiprimes; alternately, we remark that one could use Erd\"os's version of Chebyshev's theorem that the number of primes less than $x$ is at most $\log 4 \frac{x}{\log x}$ and because $\log 4 < 2$ this gives a contradiction. This completes the proof.
	
	\subsection{A comment on notation}
	Throughout this paper, we will use asymptotic notation. Since number theory, dynamics and analysis sometime use different conventions, we take a moment here to fix notation. We will write
	\begin{align*}
	x =& O(y)
	\intertext{to mean that there exists a constant $C$ such that}
	|x| \leq & C y.
	\intertext{When we adorn these symbols with subscripts, the subscripts specify which variables the constants are allowed to depend on. Thus}
	x =& O_{A,B}(y)
	\intertext{means that there exists a constant $C$ which is allowed to depend on $A$ and $B$ such that}
	|x| \leq & C y.
	\intertext{We write}
	x = & y + O(z)
	\intertext{to mean that}
	x - y = & O(z).
	\intertext{We also adopt little $o$ notation:}
	x = &o_{n \rightarrow \infty}(y)
	\intertext{means that
	\[
	\lim_{n \rightarrow \infty} \frac{x}{y} = 0.
	\]
	Occasionally, when the variable with respect to which the limit is being taken is clear from context, we may simply write
	}
	x = &o(y).
	\intertext{	As before, we write
	}
	x = & y + o_{n \rightarrow \infty}(z)
	\intertext{to mean}
	x - y = & o_{n \rightarrow \infty}(z).
	\intertext{If the expression $x$ depends on more than one variables, say $n$, $m$ and $k$, we may use subscripts to make explicit that the rate of convergence implicit in the little $o$ notation is allowed to depend on more variables. Thus,}
	x = &o_{n \rightarrow \infty, m, k}(y)
	\end{align*}
	means that $\frac{x}{y}$ tends to zero with $n$ at a rate which may depend on $m$ and $k$.
	
	\subsection{Acknowledgments}
	I would like to thank Florian Richter for his patience and conversation. I would also like to thank Terence Tao for including this proof in his class on number theory and for many helpful discussions. I would like to thank Gergley Harcos and Joni Ter\"av\"ainen for comments on an earlier draft. I would also like to thank Tim Austin, Will Baker, Bjorn Bringmann, Asgar Jamneshan, Gyu Eun Lee, Adam Lott, Clark Lyons, Bar Roytman, Chris Shriver and Will Swartworth for helpful conversations.
	
	\section{Proof of the prime number theorem}
	
	From number theory, we will use Mertens' Theorem, in particular the version which states, 
	\[
	\sum_{p \leq x} \frac{1}{p} = \log \log x + M + O\left( \frac{1}{\log x} \right)
	\]
	for some constant M; we will also use Chebyshev's Theorem, the Selberg Symmetry Formula, Landau's formulation of the prime number theorem (i.e. that the prime number theorem is equivalent to $\sum_{n \leq N} \mu(n) = o_{N \rightarrow \infty}(N)$) and a slightly modified version of the Tur\'an-Kubilius inequality which we will prove using the following Bombieri-Halász-Montgomery inequality.
	\begin{prop}[Bombieri-Halász-Montgomery inequality \cite{MR286773}]\label{pregood}
		Let $w_i$ be a sequence of nonnegative real numbers. Let $u$ and $v_i$ be vectors in a Hilbert space. Then
		\[
		\sum_{i = 1}^n w_i | \langle u, v_i \rangle |^2 \leq || u ||^2 \cdot \left( \sup_i \sum_{j = 1}^n w_j |\langle v_i, v_j \rangle| \right).
		\]
	\end{prop}
	\begin{proof}
		By duality, there exists $c_i$ such that
		\[
		\sum_{i = 1}^n w_i |c_i|^2 = 1
		\]
		and 
		\begin{align*}
		\sum_{i = 1}^n w_i | \langle u, v_i \rangle |^2  = & \left( \sum_{i = 1}^n w_i c_i \langle u, v_i \rangle \right)^2
		\intertext{and therefore by conjugate bilinearity of the inner product}
		= &  \left\langle u, \sum_{i = 1}^n w_i \overline{c}_i v_i \right\rangle^2.
		\intertext{By Cauchy-Schwarz, this is at most}
		\leq & || u ||^2 \bigg|\bigg| \sum_{i = 1}^n w_i \overline{c}_i v_i \bigg|\bigg|^2.
		\intertext{
			By the pythagorean theorem this is given by
		}
		= & || u ||^2 \sum_{i = 1}^n \sum_{j = 1}^{n} w_i w_j c_i \overline{c}_j \langle v_i, v_j \rangle.
		\intertext{
			The geometric mean is dominated by the arithmetic mean.
		}
		\leq & || u ||^2 \sum_{i = 1}^n \sum_{j = 1}^{n} w_i w_j \frac{1}{2}(|c_i| + |c_j|) | \langle v_i, v_j \rangle |.
		\intertext{
			By symmetry this is
		}	
		= & || u ||^2 \sum_{i = 1}^n  w_i |c_i|^2 \sum_{j = 1}^{n} w_j | \langle v_i, v_j \rangle |.
		\intertext{
			Because everything is nonnegative, we may replace the inner term with a supremum
		}
		\leq  & || u ||^2 \sum_{i = 1}^n  w_i |c_i|^2 \sup_k \sum_{j = 1}^{n} w_j | \langle v_k, v_j \rangle |.
		\end{align*}
		Using that $\sum w_i |c_i|^2 = 1$ completes the proof.
	\end{proof}
	The next proposition applies the previous proposition in order to show that, for any bounded sequence, the average of the sequence is the same as the average over the $p^{th}$ terms in the sequence for most prime $p$.
	\begin{prop}[Tur\'an-Kubilius \cite{MR0160745}]\label{good}
		Let $S$ denote a set of primes less than some natural number $P$. Let $N$ be a natural number which is at least $P^3$. %Let $\ell(S)$ denote
%		\[
%		\ell(S) = \sum_{p \in S} \frac{1}{p}.
%		\]
		Let $f$ be a 1-bounded function from $\N$ to $\C$.	
		Then
		\[
		\sum_{p \in S} \frac{1}{p} \left| \frac{1}{N} \sum_{n \leq N} f(n) (1 - p \ch_{p|n} ) \right|^2\ = O\left( 1 \right).
		\]
	\end{prop}
	\begin{proof}
		We will apply Proposition \ref{pregood}: our Hilbert space is $L^2$ on the space of function on the integers $\{1, \ldots, N\}$ equipped with normalized counting measure; set $w_p = \frac{1}{p}$; set $v_p = (n \mapsto 1 - p \ch_{p|n})$ and $u = f$; thus, by Proposition \ref{pregood}
		\begin{align}
		& \sum_{p \in S} \frac{1}{p} \left| \frac{1}{N} \sum_{n \leq N} f(n) (1 - p \ch_{p|n} ) \right|^2 \nonumber
		\\ \leq & \frac{1}{N} \sum_{n \leq N} |f(n)|^2 \cdot \sup_{p \in S} \sum_{q \in S} \frac{1}{q} \left| \frac{1}{N} \sum_{n \leq N} (1 - p \ch_{p|n} ) (1 - q \ch_{q|n} ) \right|. \nonumber
		\intertext{
			Since $f$ is 1-bounded, we may bound the $L^2$ norm of $f$ by 1. Thus,}
		\leq & \sup_{p \in S} \sum_{q \in S} \frac{1}{q} \left| \frac{1}{N} \sum_{n \leq N} (1 - p \ch_{p|n} ) (1 - q \ch_{q|n} ) \right|. \label{eq1}
		\end{align}
		For primes $p$ and $q$, 
			\[
			\frac{1}{N} \sum_{n \leq N} (1 - p \ch_{p|n} ) (1 - q \ch_{q|n} )
			\]
			can be expanded into a signed sum of four terms
			\[
			\frac{1}{N} \sum_{n \leq N} 1 - p \ch_{p|n} - q \ch_{q|n} + p q \ch_{p | n} \ch_{q | n}.
			\]
			When $p \neq q$, we claim that each term is $1 + O\left( \frac{P^2}{N} \right)$. The trickiest term is the last term
			\[
			\frac{1}{N} \sum_{n \leq N} p q \ch_{p | n} \ch_{q | n}.
			\]
			When $p \neq q$, we have that
			\[
			\ch_{p | n} \ch_{q | n} = \ch_{p q | n}.
			\]
			Of course, for any natural number $m$,
			\[
			\# \text{ of $n \leq N$ such that $m$ divides $n$ } = \frac{N}{m} + O(1),
			\]
			where the $O(1)$ term comes from the fact that $m$ need not perfectly divide $N$. Thus,
			\[
			\frac{1}{N} \sum_{n \leq N} p q \ch_{p | n} \ch_{q | n} = p q \left( \frac{1}{p q} + O\left( \frac{1}{N} \right) \right),
			\]
			which is $1 + O\left( \frac{P^2}{N} \right)$ as claimed. A similar argument handles the three other terms. Altogether, we conclude that 
						\[
			\frac{1}{N} \sum_{n \leq N} (1 - p \ch_{p|n} ) (1 - q \ch_{q|n} ) = O\left( \frac{P^2}{N} \right),
			\]
			when $p \neq q$. Inserting this bound into \ref{eq1} and remembering that there are at most $P$ terms in the sum over $q$ in $S$, we find
		\begin{align*}
		& \sum_{p \in S} \frac{1}{p} \left| \frac{1}{N} \sum_{n \leq N} f(n) (1 - p \ch_{p|n} ) \right|^2 
		\\ \leq & \sup_{p \in S} \sum_{q \in S} \frac{1}{q} \left| \frac{1}{N} \sum_{n \leq N} (1 - p \ch_{p|n} ) (1 - q \ch_{q|n} ) \right|
		\\ \leq & \sup_{p \in S} \frac{1}{p} \left| \frac{1}{N} \sum_{n \leq N} (1 - p \ch_{p|n} ) (1 - p \ch_{p|n} ) \right| + O\left( \frac{P^3}{N} \right)
		\intertext{Expanding out the product, the main term is}
		& \sup_{p \in S} \frac{1}{p} \left| \frac{1}{N} \sum_{n \leq N} p^2 \ch_{p|n} \right|.
		\intertext{
		By the same trick as before, we may replace the average of $\ch_{p | n}$ by $\frac{1}{p}$ plus a small error dominated by the main term. Cancelling factors of $p$ as appropriate, we are left with
		}
		= & \sup_{p \in S} \frac{1}{p} \left| \frac{1}{N} \sum_{n \leq N} p^2 \frac{1}{p} \right| = O(1).
		\end{align*}
		Of course, all the smaller terms can be bounded by the triangle inequality.
		This completes the proof.
	\end{proof}
Note that,
	\[
	\sum_{p \in S} \frac{1}{p} \frac{1}{N} \sum_{n \leq N} 1 = \sum_{p \in S} \frac{1}{p}.
	\]
	For instance, if $S$ is the set of all primes less than $P$, Euler proved that
	\[
	\sum_{p \leq P} \frac{1}{p} \rightarrow \infty
	\]
	as $P$ tends to infinity. In fact, Mertens' theorem states that this sum is approximately $\log \log P$. Thus, Proposition \ref{good} represents a real improvement over the trivial bound. Therefore, for $S$, $P$, $N$ and $f$ as in the statement of Proposition \ref{good}
\[ \left| \frac{1}{N} \sum_{n \leq N} f(n) (1 - p \ch_{p|n} ) \right|^2 \]
	is small for ``most" primes. This shows that most primes are ``good" in the sense that
	\[
	\frac{1}{N} \sum_{n \leq N} f(n) \approx \frac{1}{N} \sum_{n \leq N} f(n) p \ch_{p|n}
	\]
	This notion is captured in the following definition.
	\begin{defn}\label{defofgood}
		Let $\e$ be a positive real number, let $P$ be a natural number which is sufficiently large depending on $\e$ and let $N$ be a natural number sufficiently large depending on $P$. Denote by $\ell(N)$ the quantity
		\[
		\ell(N) = \sum_{n \leq N} \frac{1}{n}.
		\]
		Denote by $S(N)$ the set of primes $p \leq P$ such that
		\[
		\frac{1}{N} \left| \sum_{n \leq N} \mu(n) - \sum_{n \leq N} \mu(n) p \ch_{p | n} \right| \geq \e.
		\]
		Then we say a prime $p$ is good if
		\[
		\frac{1}{\ell(N)} \sum_{n \leq N} \frac{1}{n} \ch_{p \in S(n)} \leq \e.
		\]
		Otherwise, we say $p$ is bad.
	\end{defn}
From Proposition \ref{good}, we obtain the following corollary.
	\begin{cor}\label{notmany}
		Let $\e$ be a positive real number, let $P$ be a natural number which is sufficiently large depending on $\e$ and let $N$ be a natural number sufficiently large depending on $P$. Then the set of bad primes is small in the sense that
		\[
		\sum_{p \text{ bad } \leq P} \frac{1}{p} = O(\e^{-3}).
		\]
	\end{cor}
	\begin{proof}
		By Proposition \ref{good}, for each $n$ sufficiently large,
		\begin{align*}
		\sum_{p \leq P }\frac{1}{p} \ch_{p \not\in S(n)}  =& \ O(\e^{-2}).
		\intertext{
			Summing in $n$ gives,
		}
		\sum_{p \leq P } \frac{1}{p} \frac{1}{\ell(N)} \sum_{n \leq N} \frac{1}{n} \ch_{p \not\in S(n)}  =& \ O(\e^{-2}) + o_{N \rightarrow \infty, P}(1).
		\intertext{
			We remark that for $N$ sufficiently large depending on $P$, this second error term may be absorbed into the first term. By definition, the set of bad primes is the set of primes such that
		}
		\frac{1}{\ell(N)} \sum_{n \leq N} \frac{1}{n} \ch_{p \not\in S(n)} \geq& \ \e.
		\end{align*}
		But then by Chebyshev's inequality (i.e. not his theorem on counting primes),
		\[
		\sum_{p \text{ bad } \leq P} \frac{1}{p} = O(\e^{-3}).
		\]
		as desired.
	\end{proof}

	Next, we turn to the Selberg symmetry formula.
		To state Selberg's symmetry formula, we need to introduce the following function.
		Let $\Lambda_2 = \log \cdot \Lambda + \Lambda * \Lambda$ i.e.
		\[
		\Lambda_2(n) = \log(n) \Lambda(n) + \sum_{d | n} \Lambda(d) \Lambda\left( \frac{n}{d}\right),
		\]
		where the von Mangoldt function $\Lambda(n)$ when $\log p$ is $n$ is a power of a prime $p$ and $0$ otherwise. Thus, we remark that $\Lambda_2$ is supported on prime powers and products of two prime powers. It is not too hard to show that $\Lambda_2$ is ``mostly" supported on primes and semiprimes. Recall that the prime number theorem is the statement that
		\[
		\frac{1}{N} \sum_{n \leq N} \Lambda(n) = 1 + o_{N \rightarrow \infty}(1)
		\]
		and thus
		\[
		\frac{1}{N} \sum_{n \leq N} \Lambda(n) \log n = \log(N) (1 + o_{N \rightarrow \infty}(1)).
		\]
		We are now ready to state the Selberg symmetry formula.
		\begin{thm}[Selberg symmetry formula]\label{SelbergSymmetryFormula}
		The average of the second von Mangoldt function defined above is 
		\[
		\frac{1}{N} \sum_{n \leq N} \Lambda_2(n) = 2 \log N (1 + o_{N \rightarrow \infty}(1) ).
		\]
		\end{thm}
	We will refer the reader to, for instance, \cite{taoblog3} section 1 for the proof. The next proposition says that, at each scale, there are either many primes or many semiprimes.
	
	\begin{prop}\label{either}
		Let $\e > 0$ be a sufficiently small number. Suppose that $k_0$ is sufficiently large depending on $\e$ and let $I_k$ denote the interval $[(1+\e)^k, (1+\e)^{k+1}]$. Then for every $k \geq k_0$,
		\[
		\sum_{p \in I_k} \frac{1}{p} \geq \frac{1}{k}
		\]
		or
		\[
		\sum_{ \substack{p_1 p_2 \in I_{k} \\ p_i \geq  \exp(\e^3 k)} } \frac{1}{p_1 p_2} \geq \frac{1}{k}.
		\]
	\end{prop}
	\begin{proof}
		This follows from the Selberg symmetry formula (Theorem \ref{SelbergSymmetryFormula}): after all, by the Selberg symmetry formula, for $k_0$ sufficiently large, for all $k \geq k_0$,
		\begin{align}
		\frac{1}{(1+\e)^k} \sum_{n \leq (1+\e)^k } \Lambda_2(n) = & 2 \log (1+\e)^k (1 + O(\e^2) ). \nonumber
		\intertext{
		The same holds for $k$ replaced by $k + 1$.
		}
		\frac{1}{(1+\e)^{k+1}} \sum_{n \leq (1+\e)^{k+1} } \Lambda_2(n) = & 2 \log (1+\e)^{k+1} (1 + O(\e^2) ). \nonumber
		\intertext{
		Taking differences, and using that $k \log (1+ \e)  = (k+1) \log(1+\e) (1 + O(\e^2))$, for $k \geq k_0$ sufficiently large, we find that
		}
		\frac{1}{\e(1+\e)^k} \sum_{n \in I_k } \Lambda_2(n) = & 2 \log (1+\e)^k (1 + O(\e) ). \label{sseq}
		\end{align}
		We aim to show that prime powers do not contribute very much to this sum. Notice that, if a prime power contributes to the sum, then the corresponding prime must be at most the square root of $(1+\e)^{k+1}$ and there is at most one power of any prime in the interval $I_k$ (because $\e < 1$). Also, notice that $\Lambda_2(p^a) \leq 2 \Lambda(p^a) \log p^a$. Thus, we bound
		\begin{align*}
		\frac{1}{\e(1+\e)^k} \sum_{ \substack{n = p^a, a > 1 \\ n \in I_k} } \Lambda(n) \log n = &
		\frac{1}{\e(1+\e)^k} \sum_{ \substack{n = p^a, a > 1 \\ n \in I_k} }  \log p \log p^a \\ 
		\leq& 	\frac{1}{\e(1+\e)^k} \sum_{p \leq (1+\e)^{(k+1)/2} } \log p \log (1+\e)^{k+1}.
		\intertext{
		Now the number of primes less than $(1+\e)^{(k+1)/2}$ is certainly less than $(1+\e)^{(k+1)/2}$, so
		}
		\leq& 	\frac{1}{\e(1+\e)^k} (1+\e)^{(k+1)/2} \log (1+\e)^{(k+1)/2} \log (1+\e)^{k+1}. \\
		=& o_{k \rightarrow \infty, \e}(1).
		\intertext{
		For instance, by choosing $k_0$ large depending on $\e$, we can make this quantity
		}
		=& O(\e)
		\intertext{Similarly for products of a natural number $m$ and a prime power,}
		\frac{1}{\e(1+\e)^k} \sum_{ \substack{n = p^am, a > 1 \\ n \in I_k} } \Lambda(p) \Lambda(m)  \leq &
		\frac{1}{\e(1+\e)^k} \sum_{ \substack{n = p^am, a > 1 \\ n \in I_k} }  \log p \log m \\
		\leq & \frac{1}{\e(1+\e)^k} \sum_{p \leq (1+\e)^{(k+1)/2} } \log p \sum_{1 < a \leq \log_p (1+\e)^k } \sum_{ m p^a  \in I_k    } \Lambda(m). 
		\intertext{
		Now the inner most sum is bounded by Chebyshev's inequality. We simplify slightly using the factor of  $\frac{1}{\e(1+\e)^k}$ out front.
		}
		\leq & C \sum_{p \leq (1+\e)^{(k+1)/2} } \log p \sum_{1 < a \leq \log_p (1+\e)^k } \frac{1}{p^a}. \\
		\leq & C \sum_{p \leq (1+\e)^{(k+1)/2} } \frac{\log p}{p^2} . \\
		=& O(1).
		\end{align*}
		Finally, we claim that when one of the prime factors of a semiprime is less than $\exp(\e^3 k)$ then that semiprime does not contribute very much to the sum. Indeed,
		\[
		\frac{1}{\e(1+\e)^k} \sum_{ \substack{p_1 p_2 \in I_k \\ p_1 \leq \exp(\e^3 k) } } \Lambda(p_1) \Lambda(p_2) =  \frac{\e}{(1+\e)^k} \sum_{ \substack{p_1 p_2 \in I_k \\ p_1 \leq \exp(\e^3 k) } } \log p_1 \log p_2.
		\]
		Now we use that $p_1$ is at most $\exp(\e^3 k)$ and $p_2$ is at most $(1+\e)^{k+1}$.
		\begin{align*}
		\leq & \frac{1}{\e(1+\e)^k} \sum_{ \substack{p_1 p_2 \in I_k \\ p_1 \leq \exp(\e^3 k) } } \e^3 k \cdot \log(1+\e)^{k+1}.
		\intertext{ 	
		Summing over scales,
		}
		\leq & \frac{1}{\e(1+\e)^k} \e^3 k \cdot \log(1+\e)^{k+1} \sum_{m \leq k \e^3} \sum_{ \substack{  m \leq \log p_1 \leq m - 1 \\ p_1 p_2 \in I_k } } 1.
		\intertext{
		The number of terms in the inner sum is can be estimated using Chebyshev's theorem. The outersum has roughly $k \e^3$ many terms. Thus, for some constant $C$,
		}
		\leq & C \frac{1}{\e(1+\e)^k} \cdot \e^3 k \log(1+\e)^{k+1} \cdot  \e^3 k \cdot \frac{\exp(\e^3 k)}{\e^3 k} \frac{(1+\e)^{k+1}}{\log(1+\e)^{k+1}} 
		\intertext{
		Simplifying, this is
		}
		= & O(\e^2 k) \\
		= & O(\e \cdot \log (1+\e)^k ).
		\end{align*}
		Altogether, we find that we can restrict \ref{sseq} to primes and semiprimes where neither factor is too small.
		\begin{align*}
		\frac{1}{\e(1+\e)^k} \sum_{ \substack{n \in I_k \\ n = p \text{ or } n = p_1p_2 \\ p_i \geq \exp(\e^3 k) } } \Lambda_2(n) = & 2 \log (1+\e)^k (1 + O(\e) ).
		\intertext{
		For any two numbers $n$ and $m$ in $I_k$, $\frac{1}{n} = \frac{1}{m} \cdot (1 + O(\e))$, so
	}
	\e^{-1} \sum_{ \substack{n \in I_k \\ n = p \text{ or } n = p_1p_2 \\ p_i \geq \exp(\e^3 k) } } \frac{\Lambda_2(n)}{n} = & 2 \log (1+\e)^k (1 + O(\e) ).
	\intertext{
	By the pigeonhole principle, either
	}
	\e^{-1} \sum_{ p \in I_k} \frac{\Lambda_2(p)}{p} \geq & \log (1+\e)^k (1 + O(\e) )
		\intertext{or}
	\e^{-1} \sum_{ \substack{p_1p_2 \in I_k \\ p_i \geq \exp(\e^3 k) } } \frac{\Lambda_2(p_1 p_2)}{p_1 p_2} \geq & \log (1+\e)^k (1 + O(\e) ).
	\intertext{
	In the first case, moving the $\e$ and $\log p \approx \log(1+\e)^k$ terms to the other side
	}
	\sum_{ p \in I_k} \frac{1}{p} \geq &  \frac{\e}{k \log(1+\e)} \cdot (1 + O(\e) ).
	\intertext{
	Taylor explanding the logarithm gives}
	\geq &  \frac{1}{k} \cdot (1 + O(\e) ),
	\intertext{
	 as desired. In the second case,
	}
	\e^{-1} \sum_{ \substack{p_1p_2 \in I_k \\ p_i \geq \exp(\e^3 k) } } \frac{1}{p_1 p_2} k^2 \log^2(1+\e) \geq & \log (1+\e)^k (1 + O(\e) ).
	\intertext{
	Rearranging terms gives
	}
	\sum_{ \substack{p_1p_2 \in I_k \\ p_i \geq \exp(\e^3 k) } } \frac{1}{p_1 p_2} \geq & \frac{\e}{k \log (1+\e)} (1 + O(\e) ).
	\end{align*}
	Taylor expanding the logarithm again completes the proof.
	\end{proof}
	Next, we show that we can actually find two nearby scales where both inequalities from Proposition \ref{either} hold. The key idea is to use the connectedness of the interval. 
	\begin{prop}\label{both}
		Let $\e > 0$ be a number sufficiently small. Suppose that $k_0$ is sufficiently large depending on $\e$ and let $I_k$ denote the interval $(1+\e)^k$ to $(1+\e)^{k+1}$. Then there exists $k$ and $k'$ such that $|k - k'| \leq 1$ with $k$ and $k'$ in $[k_0, \e^{-2} + k_0]$ and such that
		\[
		\sum_{p \in I_k} \frac{1}{p} \geq \frac{1}{2 k}
		\]
		and
		\[
		\sum_{ \substack{p_1 p_2 \in I_{k'} \\ p_i \geq  \exp(\e^3 k')} } \frac{1}{p_1 p_2} \geq \frac{1}{2 k'}
		\]
	\end{prop}
	\begin{proof}
		Suppose not. Then by Proposition \ref{either}, for each $k$ in $[k_0, \e^{-2} +k_0]$ either
		\[
		\sum_{p \in I_k} \frac{1}{p} \geq \frac{1}{2 k}
		\]
		or
		\[
		\sum_{ \substack{p_1 p_2 \in I_{k} \\ p_i \geq  \exp(\e^3 k)} } \frac{1}{p_1 p_2} \geq \frac{1}{2 k}.
		\]
		If both hold for some $k$, then by choosing $k = k'$, we could conclude that Proposition \ref{both} holds. Thus, we will assume that exactly one of 
		\[
		\sum_{p \in I_k} \frac{1}{p} \geq \frac{1}{2 k}
		\]
		or
		\[
		\sum_{ \substack{p_1 p_2 \in I_{k} \\ p_i \geq  \exp(\e^3 k)} } \frac{1}{p_1 p_2} \geq \frac{1}{2 k}
		\]
		hold for any choice of $k$.
		Whichever holds for $k_0$ must also hold for $k_0 + 1$ since otherwise we may choose $k = k_0$ and $k' = k_0 +1$. Inductively, we may assume that for every $k$ in $[k_0, \e^{-2} + k_0]$ either
		\[
		\sum_{p \in I_k} \frac{1}{p} < \frac{1}{2 k}
		\]
		or
		\[
		\sum_{ \substack{p_1 p_2 \in I_{k} \\ p_i \geq  \exp(\e^3 k)} } \frac{1}{p_1 p_2} < \frac{1}{2 k}.
		\]	
		Summing in $k$, we eventually obtain a contradiction with Mertens' theorem: either
		\begin{equation}\label{temp1}
		\sum_{ (1+\e)^{k_0} \leq p \leq (1+\e)^{k_0 + \e^{-2}}  }\frac{1}{p} < \frac{1}{10} \cdot \left( \log(k_0 + \e^{-2}) - \log k_0 + O\left( \frac{1}{k_0} \right)    \right) 
		\end{equation}
		or
		\begin{equation}\label{temp2}
		\sum_{k \in [k_0, \e^{-2} + k_0] } \sum_{ \substack{p_1 p_2 \in I_{k} \\ p_i \geq  \exp(\e^3 k)} } \frac{1}{p_1 p_2} < \frac{1}{2} \cdot \left( \log(k_0 + \e^{-2}) - \log k_0 + O\left( \frac{1}{k_0} \right)    \right). 
		\end{equation}
		We remark that a Taylor expansion could simplify 
		\[
		\log(k_0 + \e^{-2}) - \log k_0 + O\left( \frac{1}{k_0} \right) = O\left( \frac{1}{\e^2 k_0} \right). 
 		\]
		Note that Mertens' theorem implies that
		\[
		\sum_{p \leq x} \frac{1}{p} = \log \log x + M + O\left( \frac{1}{\log x} \right),
		\]
		for some constant $M$.
		Taking differences,
		\begin{align*}
		\sum_{ (1+\e)^{k_0} \leq p \leq (1+\e)^{k_0 + \e^{-2}}  }\frac{1}{p} & = \log \log (1+\e)^{k_0 + \e^{-2}} - \log \log (1+\e)^{k_0} + O\left( \frac{1}{k_0 \log (1+\e)} \right) \\
		& = \log(k_0 + \e^{-2}) - \log(k_0) + O\left( \frac{1}{k_0 \log (1+\e)} \right).
		\end{align*}
		But \ref{temp1} says that the sum on the left is $2$ times smaller than that which gives a contradiction.
	\end{proof}
In the next proposition, we show that this implies there are nearby primes and semiprimes which are good.
	\begin{prop}\label{3good}
		Let $\e > 0$. Let $P$ be a natural number which is sufficiently large depending on $\e$. Let $N$ be a natural number which is sufficiently large depending on $P$. Then there exists $p_1$, $p_2$ and $p$ such that 
		\[
		\frac{p_1 p_2}{p} = 1 + O(\e)
		\]
		with $p_1$, $p_2$ and $p$ good in the sense of Definition \ref{defofgood} meaning $p_1, p_2$ and $p$ are not in $S(n)$ for ``most" $n \leq N$ (see Definition \ref{defofgood} for details). Furthermore, we can require that $p_1$, $p_2$ and $p$ are all greater than $\frac{1}{\e}$.
	\end{prop}
	\begin{proof}
		By Proposition \ref{both}, it suffices to show that, for some $k_0$ sufficiently large depending on $\e$ with the property that $(1 + \e)^{k_0 +\e^{-2}} \leq P$ , we have
		\begin{equation}\label{temp3}
		\sum_{ \substack{p \in [(1+\e)^{k_0}, (1+\e)^{\e^{-2} + k_0}] \\ p \text{ bad} }  }\frac{1}{p} \leq \frac{1}{10 k_0}
		\end{equation}
		and that
		\begin{equation}\label{temp4}
		\sum_{ \substack{p_1 p_2 \in [(1+\e)^k_0, (1+\e)^{\e^{-2} + k_0}] \\ p_1 \text{ bad} \\ p_1^{\e^3} \leq p_2 \leq p_1^{\e^{-3}} } } \frac{1}{p_1 p_2} \leq \frac{1}{10 k_0}.
		\end{equation}
		After all, once we have shown this, we can argue as follows: by Proposition \ref{both} there exists an interval of the form $k$ and $k'$ in $[k_0, k_0 + \e^{-2}]$ with $|k - k'| \leq 1$ for which
		\[
		\sum_{p \in [(1+\e)^k, (1+\e)^{k+1} ]} \frac{1}{p} > \frac{1}{k}
		\]
		and
		\[
		\sum_{ \substack{p_1 p_2 \in I_{k'} \\ p_i \geq  \exp(\e^3 k')} } \frac{1}{p_1 p_2} \geq \frac{1}{k'},
		\]
		where $I_k = [ (1 +\e)^k, (1+\e)^{k+1} ]$ and similarly for $I_{k'}$.
		By \ref{temp3}
		\[
		\sum_{ \substack{ p \in [(1+\e)^k, (1+\e)^{k+1} ] \\ p \text{ good} } }\frac{1}{p} > 0
		\]
		and by \ref{temp4}
		\[
		\sum_{ \substack{p_1 p_2 \in I_{k'} \\ p_i \geq  \exp(\e^3 k') \\ p_1, p_2 \text{ good}  } } \frac{1}{p_1 p_2} > 0.
		\]
		Now any good $p$ in $I_k$ and any good $p_1 p_2$ in $I_{k'}$ suffices to prove the result.
		
		Now, for the sake of contradiction, suppose first that
		\[
		\sum_{ \substack{p \in [(1+\e)^k_0, (1+\e)^{\e^{-2} + k_0}] \\ p \text{ bad} }  }\frac{1}{p} \geq \frac{1}{10 k_0}.
		\]
		Summing in $k_0 \leq \log \log P$, for $P$ large enough we get that
		\[
		\sum_{ \substack{p \leq \log N \\ p \text{ bad} } } \frac{1}{p}  \geq \frac{1}{20} \log \log \log P
		\] 
		which contradicts Corollary \ref{notmany}.
		Second, suppose that  
		\[
		\sum_{ \substack{p_1 p_2 \in [(1+\e)^k_0, (1+\e)^{\e^{-2} + k_0}] \\ p_1 \text{ bad} \\ p_1^{\e^3} \leq p_2 \leq p_1^{\e^{-3}} } } \frac{1}{p_1 p_2} \geq \frac{1}{10 k_0}.
		\]
		Summing in $k_0 \leq \log \log P$ gives, for $P$ large enough
		\[
		\sum_{ \substack{p_1 p_2  \leq \log N \\ p_1 \text{ bad} \\ p_1^{\e^3} \leq p_2 \leq p_1^{\e^{-3}} } } \frac{1}{p_1 p_2} \geq \frac{1}{20} \log\log\log P.
		\]
		For each $p_1$, by Mertens' theorem,
		\[
		\sum_{ p_1^{\e^3} \leq p_2 \leq p_1^{\e^{-3}} } \frac{1}{p_2} \leq -10 \log \e.
		\]
		By Corollary \ref{notmany}, this implies
		\[
		\sum_{ \substack{p_1 p_2  \leq \log N \\ p_1 \text{ bad} \\ p_1^{\e^3} \leq p_2 \leq p_1^{\e^{-3}} } } \frac{1}{p_1 p_2} = O\left(\e^{-3} | \log \e | \right) 
		\]
		which yields a contradiction since for $P$ large enough, $\frac{1}{20} \log \log \log P \gg \e^{-3} | \log \e |$.
	\end{proof}
	Finally, we show this implies the prime number theorem.
	\begin{thm}\label{pnt}
		The prime number theorem holds, i.e.
		\[
		\frac{1}{N} \sum_{n \leq N} \Lambda(n) = 1 + o_{N \rightarrow \infty}(1)
		\]
	\end{thm}
	\begin{proof}
		Let $\e$ be a positive real number, let $P$ be a natural number which is sufficiently large depending on $\e$ and let $N$ be a natural number sufficiently large depending on $P$.
		By Proposition \ref{3good}, there exist primes $p_1$, $p_2$ and $p$ all good and greater than $\frac{1}{\e}$ such that
		\[
		\frac{p_1 p_2}{p} = 1 + O(\e).
		\]	
		By definition of a good prime, 
		\[
		\frac{1}{M} \left| \sum_{n \leq M} \mu(n) - \sum_{n \leq M} \mu(n) p \ch_{p | n} \right| \geq \e,
		\]
		for at most a small set of $M$ (exactly how small will be spelled out shortly). In particular, let $S(M)$ denote the set of primes such that 
		\[
		\frac{1}{M} \left| \sum_{n \leq M} \mu(n) - \sum_{n \leq M} \mu(n) p \ch_{p | n} \right| \geq \e.
		\]
		Then by definition of a good prime,
		\[
		\frac{1}{\ell(N)} \sum_{M \leq N} \frac{1}{M} \ch_{p_1 \in S(M)} \ch_{p_2 \in S(M)} \ch_{p \in S(M)} = O(\e).
		\]
		Thus, we may conclude that
		\[
		\frac{1}{\ell(N)}  \sum_{M \leq N} \frac{1}{M} \frac{1}{M} \left| \sum_{n \leq M} \mu(n) - \sum_{n \leq M} \mu(n) p \ch_{p | n} \right| = O(\e).
		\]
		Since $\mu(np) = - \mu(n)$ for most $n$ (including all but those $O(\frac{1}{p}) = O(\e)$ fraction of $n$ which are not divisible by $p$), we conclude that 
		\[
		\frac{1}{\ell(N)} \sum_{M \leq N} \frac{1}{M}  \left| \frac{1}{M} \sum_{n \leq M} \mu(n) + \frac{p}{M} \sum_{n \leq M/p} \mu(n) \right| = O(\e).
		\]
		Similarly, since $p_1$ is good,
		\[
		\frac{1}{\ell(N)} \sum_{M \leq N} \frac{1}{M} \left|  \frac{1}{M} \sum_{n \leq M} \mu(n) +  \frac{p_1}{M} \sum_{n \leq M/p_1} \mu(n) \right| = O(\e).
		\]	
		By change of variables,
		\[
		\frac{1}{\ell(N)} \sum_{M \leq N} \frac{1}{M} \left| \frac{p_1}{M} \sum_{n \leq M/p_1} \mu(n) + \frac{p_1 p_2}{M} \sum_{n \leq M/p_1 p_2} \mu(n) \right| = O(\e) + O\left( \frac{\log p_1}{\log N} \right).
		\]
		By the triangle inequality and since $N$ is much larger than $p_1$,
		\[
		\frac{1}{\ell(N)} \sum_{M \leq N} \frac{1}{M} \left| \frac{p}{M} \sum_{n \leq M/p} \mu(n) + \frac{p_1 p_2}{M} \sum_{n \leq M/p_1 p_2} \mu(n) \right| = O(\e).
		\]
		But since $\frac{p_1 p_2}{p} = 1 + O(\e)$, 
		\[
		\frac{1}{\ell(N)} \sum_{M \leq N} \frac{1}{M} \left| \frac{p}{M} \sum_{n \leq M/p} \mu(n) \right| = O(\e).
		\]
		and therefore, again using that $p$ is good,
		\[
		\frac{1}{\ell(N)} \sum_{M \leq N} \frac{1}{M} \left| \frac{1}{M} \sum_{n \leq M} \mu(n) \right| = O(\e).
		\]
		This is an averaged version on the equation we want. We want that
		\[
		\left| \frac{1}{N} \sum_{n \leq N} \mu(n) \right| = O(\e),
		\]
		for all $N$ sufficiently large. Thus, we just need to prove
		\begin{lem}
		Let $\e > 0$, let $N$ be sufficiently large depending on $\e$ and suppose that
		\[
		\frac{1}{\ell(N)} \sum_{M \leq N} \frac{1}{M} \left| \frac{1}{M} \sum_{n \leq M} \mu(n) \right| = O(\e).
		\]
		Then
		\[
		\left| \frac{1}{N} \sum_{n \leq N} \mu(n) \right| = O(\e),
		\]
		\end{lem}
		To prove this we use the identity
		\[
		\mu \cdot \log = - \mu * \Lambda.
		\]
		Summing both sides up to $N$ gives
		\[
		\sum_{n \leq N} \mu(n) \log n = - \sum_{n \leq N} \sum_{d | n} \mu\left( \frac{n}{d} \right) \Lambda(d).
		\]
		Now by switching the order of summation
		\[
		= - \sum_{d \leq N} \Lambda(d) \left( \sum_{n \leq N/d} \mu(n) \right).
		\]
		If it were not for the factor of $\Lambda(d)$, this would be exactly what we want. Each $\sum_{n \leq M } \mu(n)$ for an integer $M$ occurs in this sum the number of times that $\left\lfloor \frac{N}{d} \right\rfloor = M$ where $\lfloor \cdot \rfloor$ denotes the floor which is proportional to $\frac{N}{M^2}$. The factor of $\Lambda(d)$ can be removed using the Brun-Titchmarsh inequality as follows. First, we break up the sum into different scales
		\begin{align*}
		=& - \sum_{a \in (1+\e)^\N } \sum_{\substack{d \leq N \\ a \leq d < (1+\e) a } } \Lambda(d) \left( \sum_{n \leq N/d} \mu(n) \right).
		\intertext{For all $d$ between $a$ and $(1+\e)a$, the sums $\sum_{n \leq N/d} \mu(n)$ all give roughly the same value. Therefore}
		=& - \sum_{a \in (1+\e)^\N } \sum_{\substack{d \leq N \\ a \leq d < (1+\e) a } } \Lambda(d) \left( \sum_{n \leq N/a} \mu(n) \right) \\
		 + & O\left( \sum_{\substack{ a \in (1+\e)^\N \\ a \leq N  } } \sum_{ a \leq d < (1+\e) a } \Lambda(d) \sum_{ \frac{N}{a(1+\e)} \leq n \leq \frac{N}{a} } 1 \right) 
		\intertext{ 
		First, we focus on the error term.
		}	
		& O\left( \sum_{ \substack{ a \in (1+\e)^\N \\ a \leq N  } } \sum_{ a \leq d < (1+\e) a } \Lambda(d) \sum_{ \frac{N}{a(1+\e)} \leq n \leq \frac{N}{a} } 1 \right) \\
		= & O\left( \sum_{\substack{ a \in (1+\e)^\N \\ a \leq N  } } \sum_{ a \leq d < (1+\e) a } \Lambda(d)  \left(\frac{N}{a} - \frac{N}{a(1+\e)}\right) \right)	 \\
		= & O\left( \sum_{\substack{ a \in (1+\e)^\N \\ a \leq N  } } \sum_{ a \leq d < (1+\e) a } \Lambda(d)  \frac{N \e}{a (1+\e)} \right)
		\intertext{	
		By the Brun-Titchmarsh inequality	
		}
		= & O\left( \sum_{\substack{ a \in (1+\e)^\N \\ a \leq N  }} \e a \frac{N \e}{a (1+\e)} \right) \\
		= & O\left( \sum_{\substack{ a \in (1+\e)^\N \\ a \leq N  } } N \e^2 \right) \\
		= & O\left( N \e^2 \log_{1+\e} N \right) \\
		= & O\left( \e N  \log N \right)
		\intertext{
		where the last step involves Taylor expanding $\log(1 + \e)$ near $\e = 0$. Next, we turn our attention to the main term. We begin by pulling out the sum over $\mu(n)$ which no longer depends on $d$.
		}
		& - \sum_{a \in (1+\e)^\N } \sum_{\substack{d \leq N \\ a \leq d < (1+\e) a } } \Lambda(d) \left( \sum_{n \leq N/a} \mu(n) \right) \\
		=& - \sum_{a \in (1+\e)^\N } \left( \sum_{n \leq N/a} \mu(n) \right) \left( \sum_{\substack{d \leq N \\ a \leq d < (1+\e) a } } \Lambda(d) \right).
		\intertext{By the Brun-Titchmarsh inequality, this is bounded in absolute value by}
		\leq&  \sum_{\substack{a \in (1+\e)^\N \\ a \leq N} } \left| \sum_{n \leq N/a} \mu(n) \right| \cdot \left( 10 \e a \right).
		\intertext{	
		Earlier, we replaced a sum indexed by $n \leq N/d$ by a sum indexed by $n \leq N/a$, showing these two sums were close up to an error of size $O(\e N \log N)$. Undoing this process, we find
		}
		= & 10 \sum_{d \leq N} \left| \sum_{n \leq N/d} \mu(n) \right| + O(\e N \log N).
		\intertext{Now we let $M = \frac{N}{d}$. The number of values of $d$ such that $\left\lfloor \frac{N}{d} \right\rfloor$ is the number of values of $d$ such that $M \leq \frac{N}{d} < M + 1$ and therefore $\frac{N}{M+1} < d \leq \frac{N}{M}$. The number of such $d$'s is bounded by $\frac{N}{M} - \frac{N}{M+1} = \frac{N}{M(M+1)}$. Thus  }
		\leq& 10 \sum_{M \leq N} \frac{N}{M^2}\left| \sum_{n \leq M} \mu(n) \right| + O(\e N \log N).
		\intertext{But we already showed that this sum is bounded by }
		=& O(\e N \ell(N) ) \\ =& O( \e N \log N ).
		\end{align*}
		Thus,
		\[
		\sum_{n \leq N} \mu(n) \log(n) = O(\e N \log N).
		\]
		Since $\log n = \log N (1 + O(\e))$ for $n$ between $\e \frac{N}{\log N}$ and $N$ and $\e$ sufficiently small we conclude that 
		\[
		\sum_{n \leq N} \mu(n) = O(\e N).
		\]
		But this classically implies the prime number theorem.
	\end{proof}

	\section{In what ways is this a dynamical proof?}
	
	To begin the argument, we showed that for all $N$, for most $p$ i.e. all $p$ outside a bad set where
	\[
	\sum_{p \text{ bad}} \frac{1}{p} \leq C_\e
	\]
	we have that
	\[
	\sum_{n \leq N} \mu(n) = \sum_{n \leq N} \mu(n) p \ch_{p | n} + O(\e).
	\]
	We did this using an $L^2$ orthogonality argument (Propositions \ref{pregood} and \ref{good}). Alternately, we can argue using a variant of Tao's entropy decrement argument (the first version of this argument appeared in \cite{MR3569059}; a different version of the entropy decrement argument appeared in \cite{TT1} and \cite{TT2}; the version presented here is somewhat different from what appeared in those papers). Let $\n$ be a random integer less than $N$. Let $\x_i = \mu(\n + i)$ and let $\y_p = \n \mod p$. 
	In probability and dynamics, a stochastic process is a sequence of random variables  $(\ldots, \xi_{-2}, \xi_{-1}, \xi_0, \xi_1, \xi_2, \ldots)$ such that
	\[
	\P ((\xi_1, \ldots \xi_k) \in A) = \P ((\xi_{1+m}, \ldots \xi_{k+m}) \in A)
	\]
	for any set $A$ and for any $m$.
	In our setting $(\ldots, \x_{-2}, \x_{-1}, \x_0, \x_1, \x_2, \ldots)$ is approximately stationary in the sense that
	\[
	\P ((\x_1, \ldots \x_k) \in A) \approx \P ((\x_{1+m}, \ldots \x_{k+m}) \in A)
	\]
	where the two terms differ by some small error which is $o_{N \rightarrow \infty, m}(1)$. A stationary process is the same as a random variable in a measure preserving system where $\xi_{i+1}$ is the transformation applied to $\xi_i$. A key invariant of a stationary process is thus the Kolmogorov-Sinai entropy:
	\[
	h(\xi) = \lim_{n \rightarrow \infty} \frac{1}{n} H(\xi_1, \ldots, \xi_n)
	\]
	where
	\[
	H(\xi_1, \ldots, \xi_n)
	\]
	is the Shannon entropy of $(\xi_1, \ldots, \xi_n)$.
	This limit exists because 
	\begin{align*}
	\frac{1}{n} H(\xi_1, \ldots, \xi_n) =& \frac{1}{n} \sum_{i \leq n} H(\xi_i | \xi_1 \ldots, \xi_{i-1})
	\intertext{
		by the chain rule for entropy, which is equal to 
	}
	=& \frac{1}{n} \sum_{i \leq n} H(\xi_0 | \xi_{-1} \ldots, \xi_{-i+1})
	\end{align*}
	by stationarity. This is a Caesar\'o average of a decreasing sequence which is therefore decreasing. Since entropy is nonnegative, we can conclude that the limit exists. In our case, because $(\ldots, \x_{-1}, \x_0, \x_1, \ldots)$ is almost stationary, we can conclude that
	\[
	\frac{1}{n} H(\x_1, \ldots, \x_n)
	\]
	is almost decreasing in the sense that, for $m > n$,
	\[
	\frac{1}{m} H(\x_1, \ldots, \x_m) \leq \frac{1}{n} H(\x_1, \ldots, \x_n) + o_{N \rightarrow \infty, n}(1).
	\]
	The same is true for the relative entropy
	\[
	\frac{1}{n} H(\x_1, \ldots, \x_n | \y_{p_1}, \ldots, \y_{p_k})
	\]
	for any fixed set of primes $p_1, \ldots, p_k$. 
	
	We define the mutual information between two random variables $\x$ and $\y$ as
	\[
	I(\x; \y) = H(\x) - H(\x | \y)
	\]
	and more generally the conditional mutual information
	\[
	I(\x; \y | \z) = H(\x | \z) - H(\x | \y, \z).
	\]
	We assume for the rest of the explanation that all random variables take only finitely many values.
	Mutual information measures how close two random variables are to independent. Two random variables $\x$ and $\y$ are independent if and only if 
	\[
	I(\x; \y) = 0.
	\]
	Intuitively, we think of $\x$ and $\y$ as close to independent if the mutual information is small. The crux of the entropy decrement argument is that we can find primes $p$ such that $(\x_1, \ldots, \x_p)$ is close to independent of $\y_p$. The argument is as follows. Let $p_1 < p_2 < \ldots < p_k$ be a sequence of primes. Consider the relative entropy
	\begin{align*}
	& \frac{1}{p_k} H(\x_1, \ldots, \x_{p_k} | \y_{p_1}, \ldots, \y_{p_k}) \\
	= & \frac{1}{p_k} H(\x_1, \ldots, \x_{p_{k}} | \y_{p_1}, \ldots, \y_{p_{k-1}}) - \frac{1}{p_k} I(\x_1, \ldots, \x_{p_{k}}; \y_{p_{k}} | \y_{p_1}, \ldots, \y_{p_{k-1}} )
	\intertext{
		and because the relative entropy is almost decreasing
	}
	= & \frac{1}{p_{k-1}} H(\x_1, \ldots, \x_{p_{k-1}} | \y_{p_1}, \ldots, \y_{p_{k-1}}) - \frac{1}{p_k} I(\x_1, \ldots, \x_{p_{k}}; \y_{p_{k}} | \y_{p_1}, \ldots, \y_{p_{k-1}} ) + o(1).
	\intertext{
		Inductively, we find
	}
	\leq & H(\x_1) - \sum_{j \leq k} \frac{1}{p_j} I(\x_1, \ldots, \x_{p_{j}}; \y_{p_{j}} | \y_{p_1}, \ldots, \y_{p_{j-1}} ) + o(1)
	\end{align*}
	We conclude that the set of bad primes $p_j$ for which
	\[
	I(\x_1, \ldots, \x_{p_{j}}; \y_{p_{j}} | \y_{p_1}, \ldots, \y_{p_{j-1}} ) \geq \e
	\]
	satisfies
	\[
	\sum_{p_j \text{ bad}} \frac{1}{p_j} \leq \e^{-1} H(\x_1) + o(1) < \infty.
	\]
	Thus, for most primes, 
	\[
	I(\x_1, \ldots, \x_{p_{j}}; \y_{p_{j}} | \y_{p_1}, \ldots, \y_{p_{j-1}} ) < \e.
	\]
	In a slight abuse of terminology, we say such primes are good.
	Although this definition is apparently different from Definition \ref{defofgood}, we will show that this notion of good meaning small mutual information essentially implies the ``random sampling" version defined in Definition \ref{defofgood}.
	
	Intuitively, if $p$ is good then $\x_1, \ldots, \x_p$ and $\y_p$ are nearly independent. This is formalized by Pinsker's inequality. Pinsker's inequality states that
	\[
	d_{TV}(\x, \y) \leq D(\x || \y)^{1/2}
	\]
	where $d_{TV}$ is the total variation distance and $D$ is the Kullback-Leibler divergence. For our purposes, the important thing about the Kullback-Liebler divergence is that if $\y'$ is a random variable with the same distribution as $\y$ which is independent of $\x$ then 
	\[
	D( (\x,\y) || (\x, \y')  ) = I(\x; \y).
	\]
	Therefore, we conclude that
	\[
	d_{TV}( (\x, \y), (\x, \y') ) \leq  I(\x; \y)^{1/2}.
	\]
	Similarly, there is a relative version
	\[
	d_{TV}( (\x, \y, \z), (\x, \y', \z) ) \leq  I(\x; \y | \z)^{1/2},
	\]
	where now $\y'$ has the same distribution as $\y$ but is relatively independent of $\x$ over $\z$ meaning that  
	\[
	\P(\x \in A, \y \in B | \z = c) = \P(\x \in A | \z =c) \P( \y \in B | \z =c).
	\]
	Thus, for bounded function $F$,
	\[
	\E F(\x, \y, \z) = \E F(\x, \y', \z) + O(I(\x; \y)^{1/2}),
	\]
	where again $\y'$ is relatively independent of $\x$ over $\z$ and $\E$ denotes the expectation.
	In our case, for a good prime $p$ where
	\[
	I(\x_1, \ldots, \x_{p}; \y_{p} | (y_q)_{q < p} ) < \e
	\]
	we note that
	\[
	\E F(\x_1, \ldots, \x_p, \y_p) = \E F(\x_1,\ldots \x_p, \y_p') + O(\e^{1/2}).
	\]
	for any bounded function $F$ where $\y_p'$ is relatively independent of $(\x_1, \ldots, \x_p)$ over $(\y_q)_{q < p}$.
	Since  $\y_p$ and $(\y_q)_{q < p}$ are already very nearly independent by the Chinese remainder theorem (and in fact if $N$ is a multiple of the product of primes less than $p$, then $\y_p$ and $(\y_q)_{q < p}$ are genuinely independent) we can conclude that
	\[
	\E F(\x_1, \ldots, \x_p, \y_p) = \E F(\x_1,\ldots \x_p, \y_p') + O(\e^{1/2}),
	\]
	where now $\y_p'$ is genuinely independent of $(\x_1, \ldots, \x_p)$.
	For example, if we want to evaluate 
	\[
	\frac{1}{N} \sum_{n \leq N} \mu(n)
	\]
	we could interpret this as
	\[
	\E F(\x_0)
	\]
	where $F(x) = x$. Alternately, we can average
	\[
	\frac{1}{N} \sum_{n \leq N} \mu(n) \approx \frac{1}{p} \sum_{i \leq p} \mu(n + i),
	\]
	which is 
	\[
	\E F(\x_1, \ldots, \x_p)
	\]
	where now $F(x_1, \ldots, x_p) = \frac{1}{p} \sum_{i \leq p} x_i$. Now let $\y_p'$ as before be independent of $(\x_1, \ldots, \x_p)$ and uniformly distributed among residue classes mod $p$. Then this is also
	\[
	\E F(\x_1, \ldots, \x_p, \y_p')
	\]
	where 
	\[
	F(x_1, \ldots, x_p, y_p) = \frac{1}{p} \sum_{i \leq p} x_i p \ch_{y_p = -i}.
	\]
	As we noted, for $p$ a good prime, this is approximately,
	\[
	\E F(\x_1, \ldots, \x_p, \y_p') \approx \E F(\x_1, \ldots, \x_p, \y_p)
	\]
	and unpacking definitions this is
	\begin{align*}
	\E F(\x_1, \ldots, \x_p, \y_p) = &\frac{1}{N} \sum_{n \leq N} \frac{1}{p} \sum_{i \leq p} \mu(n + i) p \ch_{n = -1 \mod p}.
	\intertext{
		 Undoing the averaging in $i$ gives
	}
	\approx &\frac{p}{N} \sum_{n \leq N} \mu(n) \ch_{p | n}.
	\end{align*}
	Thus, the analogue of Corollary \ref{notmany} can be proved using the entropy decrement argument, which can be interpreted in the dynamical setting.
	
	The rest of the proof can also be translated to the dynamical setting. The Furstenberg system corresponding to the M\"obius function can be constructed as follows. The underlying space is the set of functions from $\Z$ to $\{-1, 0, 1\}$. We construct a random variable on this space. Consider a random shift of the M\"obius function. Formally, let $\n$ be a uniformly chosen random integer between 1 and $N$ and let $\X_N$ denote the function $\mu$ (say extended by $0$ to the left) shifted by $\n$ i.e. $\X_N(i) = \mu(i + \n)$. Since the underlying space of functions from $\Z$ to $\{-1, 0, 1\}$ is compact, there is a subsequence of $(\X_N)_N$ which converges weakly to a random variable $\X$. Since the distribution of each random variable $\X_N$ is ``approximately" shift invariant, the distribution of the limit $\X$ is actually shift invariant. Thus, we obtain a shift invariant measure $\nu$ on the space of functions from $\Z$ to $\{-1, 0 , 1\}$ with the property that if $f$ is the ``evaluation at zero" map 
	\[
	f( (a_n)_{n \in \Z} ) = a_0
	\]
	then 
	\[
	\int f(x) \nu(dx) = \E f(\X)
	\]
	is a subsequential limit of terms of the form
	\[
	\frac{1}{N} \sum_{n \leq N} \mu(n).
	\]
	Thus, we can encode questions about the average of $\mu$ or more generally shifts like $\mu(n) \mu(n+1)$ in a dynamical way.
	
	In order to take advantage of the fact that $\mu$ is multiplicative, we need to impose extra structure on the dynamical systems we associate to $\mu$. This extra structure is implicit in \cite{TT1} and \cite{TT2} and is explicitly described first in \cite{taoblog}. See also \cite{Sawin} and \cite{me}. One key feature of multiplicative functions is that they are statistically multiplicative in the sense that for any $\epsilon_1, \ldots, \epsilon_k$ in $\{-1, 0,1 \}$,
	\begin{align*}
	&\frac{p}{N} \# \{ n \leq N \colon \mu(n + pi) = \epsilon_i \text{ for all $i$ and } p | n \} \\ = &\frac{p}{N} \# \{ n \leq N/p \colon \mu(n + i) = -\epsilon_i \text{ for all $i$} \} + O\left( \frac{1}{p} \right).
	\end{align*}
	(This holds simply by changing variables and using that $\mu$ is  multiplicative).
	For $N$ in some subsequence, we can think of the right hand side as
	\[
	\frac{p}{N} \# \{ n \leq N/p \colon \mu(n + pi) = \epsilon_i \text{ for all $i$} \} \approx \nu \{ x \colon f(T^{i p} x) = \epsilon_i \}.
	\]
	We would like a way of encoding this identity in our dynamical system. One solution is to use logarithmic averaging. Now let $\n$ denote a random integer between $1$ and $N$ which is not uniformly distributed but which is logarithmically distributed meaning the probability that $\n  = m$ is proportional to $\frac{1}{m}$ for $m \leq N$. Let $\X_N(i) = \mu(n + i)$ be a random translate of the M\"obius function. Consider the pair $(\X_N, \n)$ in the space of pairs of functions from $\Z$ to $\{-1, 0, 1\}$ and profinite integers. This product space is compact so there is a weak limit $(\X, \y)$ where $\X$ is a functions from $\Z$ to $\{-1, 0, 1\}$ and $\y$ is a profinite integer. Let $T(x, y) = (n \mapsto x(n+1), y + 1)$. Let $\rho$ be the distribution of $(\X, \y)$ which is a $T$-invariant measure on our space. Consider the map $I_p$ on pairs of functions and profinite integers which are $0 \mod p$ which dilates the function by $p$, multiplies the function by $-1$ and divides the profinite integer by $p$ i.e.
	\[
	I_p (x, y) = (n \mapsto -x(pn), y/p ).
	\]
	For a point $(x, y)$ in our space, let $M$ denote the projection onto the second factor
	\[
	M(x, y) = y.
	\]
	Let $f$ be the ``evaluation of the function at 0" function i.e.
	\[
	f(x,y) = x(0).
	\]
	Then the dynamical system has the following properties, where $x$ is always a function from $\Z$ to $\{-1, 0 , 1\}$, $p$ and $q$ are primes and $y$ is a profinite integer:
	\begin{enumerate}
		\item For all $p$, for all $x$ and $y$ such that $M(x,y) = 0 \mod p$,
		\[
		I_p( T^p(x, y)  ) = T(I_p(x, y)).
		\]
		\item For all $p$ and $q$, for all $x$ and $y$ where $M(x,y)$ is $0 \mod p q$, we have
		\[
		I_p (I_q (x, y) ) = I_q (I_p (x,y)).
		\]
		\item For all $p$, and for all measurable functions on our space $\phi$,
		\[
		\int \phi(x, y) \rho(dx dy) = \int p \ch_{ M(x,y) = 0 \mod p } \phi( I_p(x, y) ) \rho(dx dy) + O\left( \frac{1}{p} \right).
		\] 
		\item For all $p$ and for all $x$ and $y$ such that $M(x,y) = 0 \mod p$   we have that
		\[
		f(I_p(x,y)) = -f(x,y).
		\]
	\end{enumerate}
	A tuple $(X, \rho, T, f, M, (I_p)_{p})$ where $(X, \rho, T)$ is a measure preserving system and satisfying $(1)$ through $(4)$ is a called a dynamical model for $\mu$. Translating our argument over to the dynamical context, there exists some $p$ such that
	\[
	\int f(x,y) \rho(dx dy) \approx \int f(x, y) \cdot p \ch_{M(x,y) = 0 \mod p },
	\]
	with an error term which we may make arbitrarily small by increasing $p$.
	On the other hand,
	\begin{align*}
	\int f(x, y) \cdot p \ch_{M(x,y) = 0 \mod p} = & \int -f( I_p(x, y)) \cdot p \ch_{M(x,y) = 0 \mod p} \\
	& = -\int f(x, y).
	\end{align*}
	We conclude that
	\[
	\int f = 0,
	\]
	for any dynamical model for $\mu$. 
	
	In \cite{taoblog}, Tao  constructs a dynamical model where
	\[
	\int f \approx \frac{1}{\log N} \sum_{n \leq N} \frac{1}{n} \mu(n)
	\]
	i.e. using logarithmic averaging and the Furstenberg correspondence principle.
	However using either Corollary \ref{notmany} or a version of the entropy decrement argument, we can argue as follows.
	Let $\rho_N$ denote the distribution of $(\X_N, \n)$ in the space of pairs of functions $\Z \rightarrow \{-1,0,1\}$ and profinite integers and where $\n$ is a uniformly distributed random integer between $1$ and $N$ and $\X_N(i) = \mu(\n + i)$. For any $\epsilon$ in $S^1$ and $\phi$, define $\epsilon_* \rho_n$ by 
	\[
	\int  \phi(x,y) \epsilon_* \rho_N(dx dy) = \int  \phi(\epsilon \cdot x,y) \rho_N(dx dy).
	\]
	Choose $\epsilon_N$ so that 
	\[
	\nu_m = \left( \sum_{n \leq m} \frac{1}{n} \right)^{-1}  \sum_{N \leq m} \frac{1}{N} (\epsilon_N)_*  \rho_N,
	\]
	satisfies
	\[
	\int f(x,y) \nu_m(x,y) = \left( \sum_{n \leq m} \frac{1}{n} \right)^{-1} \sum_{N \leq M} \frac{1}{N} \left| \frac{1}{N} \sum_{n \leq N} \mu(n) \right|,
	\]
	i.e. $\epsilon_N$ is the sign of $\sum_{n \leq N} \mu(n)$. Using a version of Corollary $\ref{notmany}$ or the entropy decrement argument, one can prove that for most $p$ (except for a set of logarithmic size at most a constant depending on $\e$), 
	\[
	(I_p)_* ( p \ch_{M = 0 \mod p} \ \nu_m) \approx \nu_m + O\left( \e +\frac{\log p}{\log m} \right).
	\]
	By the argument from before (see the proof of Theorem \ref{pnt}), this is enough to conclude the prime number theorem.


\end{document}